\documentclass[11pt, twosides]{amsart}
\usepackage{amsfonts}
\usepackage{indentfirst,amsthm,bm,amsmath}
\usepackage{hyperref}
\usepackage[pagewise]{lineno}
\makeatletter
\newtheorem{theorem}{Theorem}[section]

\newtheorem{lemma}{Lemma}[section]

\newtheorem{definition}{Definition}[section]
\newtheorem{remark}{Remark}[section]

\allowdisplaybreaks[4]

\begin{document}
\title[Gauss curvature estimate and Modified defect relation]
{An estimation of the Gauss curvature  and the modified defect relation for the Gauss map of
immersed harmonic surfaces in $\mathbb{R}^n$}

\author{Zhixue Liu}
\address{School of Science, Beijing University of Posts and Telecommunications, Beijing 100876, P. R. China;
Key Laboratory of Mathematics and Information Networks (Beijing University of Posts and Telecommunications), Ministry of Education, China.}
\email{zxliumath@bupt.edu.cn(the corresponding author)}
\author{Yezhou Li}
\address{School of Science, Beijing University of Posts and Telecommunications, Beijing 100876, P. R. China;
	Key Laboratory of Mathematics and Information Networks (Beijing University of Posts and Telecommunications), Ministry of Education, China.}
\email{yezhouli@bupt.edu.cn}
\thanks{This work was supported by the National Natural Science
Foundation of China (Grant No.12101068, No.12171050, No.12261106).}

\date{}

\keywords{Harmonic surfaces; $K$-quasiconformal; Gauss map; Curvature estimate; Defect relation}

\subjclass[2020]{Primary 32H25, 53A10; secondary 53C42, 30C65}
\begin{abstract}
In this paper,
we  study the estimation of Gauss curvature  for $K$-quasiconformal harmonic surface in ${\mathbb R}^3$
and present  an accurate improvement of the previous result in [6, Theorem 5.2].
Let $X:M\rightarrow{\mathbb R}^3$ denote a $K$-quasiconformal harmonic surface and
let $\mathfrak{n}$ be the unit normal map of $M$.
We define $d(p)$ as the distance from point $p$ to the boundary of $M$ and $\mathcal{K}(p)$ as the Gauss curvature of $M$ at $p$.
Assuming that the Gauss map (i.e., the normal $\mathfrak{n}$) omits $7$ directions $\mathbf{d}_1,\cdots,\mathbf{d}_7$ in $S^2$ with the property that
any three of these directions
are not contained in a plane in ${\mathbb R}^3$.
Then there exists a positive constant $C$ depending only on $\mathbf{d}_1,\cdots,\mathbf{d}_7$ such that
\begin{equation*}
|\mathcal{K}(p)|\leq C/d(p)^2
\end{equation*}
for all points $p\in M$.
Furthermore,
a  modified defect relation for the generalized Gauss map
of the immersed harmonic surfaces in $\mathbb{R}^n(n\geq 3)$ is verified.
\end{abstract}

\maketitle
\section{Introduction and main results}
The classical Bernstein theorem states that a minimal graph over the whole plane, defined as a surface that attains the smallest area for a given boundary, is planar. Consequently, its Gauss map(its unit normal) will omit at least half of the sphere.
R. Osserman \cite{OR-1961} extended this theorem to encompass surfaces that are not necessarily the graph of a function.
He demonstrated that the Gauss map of a complete minimal surface immersed in three-dimensional Euclidean space, denoted as
$\mathbb{R}^3$, cannot omit a set of positive logarithmic capacity unless the surface is flat.
In 1981, F. Xavier \cite{XA-1981} further advanced this research by proving that if a nonflat complete minimal surface
$M$ exists, the Gauss map of
$M$ can omit at most six points on the sphere.
In 1988, H. Fujimoto \cite{Fujimoto-1988} gave the following curvature estimate
for minimal surfaces immersed in ${\mathbb R}^3$.
\begin{theorem}\cite{Fujimoto-1988}\label{Fujimoto-1988}
Let $X:M\rightarrow{\mathbb R}^3$ be a nonflat noncomplete minimal surface and let
$g:M\rightarrow{\mathbb P}^1({\Bbb C})$ be the Gauss map.
We define $d(p)$ as the distance from $p$ to the boundary of $M$ and $\mathcal{K}(p)$ as the Gauss curvature of $M$ at $p$.
If $g$ omits at least five points $\alpha_1,\cdots, \alpha_5$,
then there exists a positive constant $C$ that depends only on
$\alpha_1,\cdots, \alpha_5$ such that
$$|\mathcal{K}(p)|\leq C/d(p)^2$$
for any point $p\in M$.
\end{theorem}
Building upon these findings, H. Fujimoto \cite{Fujimoto-1988} obtained a precise result by proving that the maximum number of exceptional values omitted by the Gauss map of such a surface is four.
Notably, there exists a complete minimal surface for which the Gauss map precisely omits four directions (see \cite{OR-1961}).
The results mentioned above reveal that the Gauss map of a complete minimal surface immersed in
${\mathbb R}^3$ shares many similar value distribution properties  with meromorphic functions defined on ${\mathbb C}$.\par

\medskip
In the late 1960s,
T. K. Milnor \cite{KT-1967,KT-1968}
initiated the investigation of
whether the theory of minimally immersed surfaces
extends in an interesting manner to the
broader class of harmonically immersed surfaces.
Subsequently,
many results have been established in this area
(see \cite{AL-2013,CKW-2015,CKW-2018,DS-2019,JR-1988,JR-1989,KD-2013,KT-1979,KT-1980,KT-1983}).
As we know, minimal surfaces immersed in $\mathbb{R}^n$  are enormously special, they are some conformal harmonic surfaces.
However,
not all  properties of minimal surfaces  have reasonable counterparts in the case of harmonic surfaces.
For instance, the classical Bernstein theorem  failed
to hold if we only assume the graph is harmonic.
The geometric properties of a surface are largely influenced by the choice of metrics on the surface.
Unlike the case of  minimal surfaces,
the metric $ds^2$(see (\ref{pro1.1.}))  on the immersed harmonic surfaces $M$
induced from the standard inner product on $\Bbb{R}^n$ does not need to be conformal metric.
Instead, it consists of two parts: {\it the conformal metric $\Gamma$}
and {\it the Hopf differential} (see the next section for the details).
\par

\medskip
In a recent study,
the author and collaborators \cite{Chen-2021}
extended the value distribution theory on minimal surfaces to the larger class of
$K$-quasiconformal harmonic surfaces(the definition of $K$-quasiconformal harmonic surfaces will be given in the next section).
They proved that the unit normal map $\mathfrak{n}$ of a nonflat complete $K$-quasiconformal harmonic surfaces
omits at most 6 directions in $S^2$
that satisfy any three of which  are not contained in a plane in ${\Bbb R}^3$.
Furthermore, they conducted further research on estimating the curvature of K-quasiconformal harmonic surfaces when the normal $\mathfrak{n}$ omits a small neighborhood around a fixed direction.

\begin{theorem}\cite{Chen-2021}\label{Chen-2021}
Let $X:M\rightarrow{\mathbb R}^3$ be a $K$-quasiconformal harmonic surface in $\mathbb{R}^3$.
Suppose that the normals to the points of $M$ all make an angle of at least $\theta>0$ with some fixed direction and satisfying that $|(\phi'\cdot\phi)(\overline{\phi'}\cdot\phi)|/\|\phi\|^4\leq N_K$, where $\phi=\partial X/\partial z$ and $N_K$ is a constant.
Then we have the following inequality:
\begin{equation}\label{thm1.1}
|\mathcal{K}(p)|\leq C/d(p)^2
\end{equation}
where $C$ is a constant that depends  on $K,\theta, N_K$.
\end{theorem}
\begin{remark}
The condition ``$|(\phi'\cdot\phi)(\overline{\phi'}\cdot\phi)|/\|\phi\|^4\leq N_K$"
in Theorem \ref{Chen-2021} can be removed for the case of minimal surfaces. In other words,  it is automatically satisfied.
This can be seen by considering a minimal immersion
 $X$, where $h=\phi\cdot\phi=0$.
Taking differential on  both sides of $\phi\cdot\phi=0$,  we obtain $\phi'\cdot\phi=0$.
\end{remark}
It is worth noting that Theorem \ref{Chen-2021}
requires
the normal vectors of $M$  to
 form a minimum angle of
$\theta>0$  with a fixed direction,
which implies that the normal $\mathfrak{n}$  omits a certain neighborhood of certain directions.
In this paper,
we delves further into the study of Gauss curvature estimates for $K$-quasiconformal harmonic surface in ${\mathbb R}^3$
and present  an accurate improvement of  Theorem \ref{Chen-2021} by establishing that
(\ref{thm1.1}) holds when  the normal
$\mathfrak{n}$ omits 7 directions such that any three of these  directions are not contained in a plane in $\mathbb{R}^3$.
Importantly,
we demonstrate that the condition ``$|(\phi'\cdot\phi)(\overline{\phi'}\cdot\phi)|/\|\phi\|^4\leq N_K$" is not necessary  for the case of $K$-quasiconformal harmonic surfaces.
One of our main results is  stated  as follows.
\begin{theorem}\label{mainthm-2}
Let $X:M\rightarrow{\mathbb R}^3$ be a $K$-quasiconformal harmonic surface and
$\mathfrak{n}$ be the unit normal map of $M$.
If its Gauss map (i.e., the normal $\mathfrak{n}$) omits $7$ directions $\mathbf{d}_1,\cdots,\mathbf{d}_7$ in $S^2$
such that no three of them are contained in a plane in ${\mathbb R}^3$,
then there exists a positive constant $C$ that depends only on $\mathbf{d}_1,\cdots,\mathbf{d}_7$ satisfying the following inequality for all $p\in M$:
\begin{equation}\label{Curvature}
|\mathcal{K}(p)|\leq C/d(p)^2.
\end{equation}
\end{theorem}
If the surface $M$ is complete,
then for any $p\in M$,
$d(p)\equiv\infty$, and (\ref{Curvature}) holds.
Consequently,
 Theorem  \ref{mainthm-2} implies that
the unit normal map $\mathfrak{n}$ of a nonflat complete
 $K$-quasiconformal harmonic surface
omits at most 6 directions in $S^2$ such that any three of these directions
are not contained in a plane in ${\mathbb R}^3$,
thereby  recovering \cite[Theorem 4.7]{Chen-2021}.\par

\medskip
In addition to our investigation on Gauss curvature estimate, we also explore the value distribution properties of Gauss map for  immersed harmonic surfaces in $\mathbb{R}^n$.
This study was initiated by R. Osserman and S. S. Chern \cite{Chern-1965,CO-1967,OR-1964}
who focused on the case of complete minimal surfaces immersed in ${\mathbb R}^n$.
Over the past few decades, many important work has been made in this area(see \cite{Fujimoto-1990,OS-1997,RU-1991}).
For instance,
H. Fujimoto \cite{Fujimoto-1990} has shown that the Gauss map of  nonflat minimal surfaces in $\mathbb{R}^n$
can omit at most $n(n+1)/2$ hyperplanes in general position in ${\mathbb P}^{n-1}(\mathbb C)$, assuming that  $G$ is nondegenerate.
M. Ru\cite{RU-1991} later removed
the ``nondegenerate" assumption.
Inspired by those work on minimal surfaces,
we give the modified defect(see Definition \ref{def-1}) relation for the  Gauss map
of immersed harmonic surfaces in $\mathbb{R}^n$.
\begin{theorem}\label{mainthm-1}
Let $X:M\rightarrow \mathbb{R}^n$ be a nonflat weakly complete
harmonic surface with the induced metric, where $M$ is an open Riemann surface
and $G:M\rightarrow{\mathbb P}^{n-1}({\mathbb C})$ be the generalized Gauss map.
If $G$ is $k-$nondegenerate, where $1\leq k\leq n-1$,
then for any  hyperplanes $H_1,\cdots, H_q\in{\mathbb P}^{n-1}({\mathbb C})$ in general position,
we have the following inequality:
$$\sum_{j=1}^q\delta_G^H(H_j)\leq (2n-k-1)(k/2+1).$$
\end{theorem}
\begin{remark}
It is worth noting that in the case of a nonflat weakly complete harmonic surface $X:M\rightarrow \mathbb{R}^n$,
Theorem \ref{mainthm-1} guarantees  that the generalized Gauss map $G$ omits at most
$(2n-k-1)(k/2+1)$ hyperplanes in general position in ${\mathbb P}^{n-1}(\mathbb C)$.
This result extends previous findings on minimal surfaces to the case of more general harmonic surfaces. Furthermore,
 if the generalized Gauss map $G$ omits more than $n(n+1)/2$ hyperplanes in general position in ${\mathbb P}^{n-1}(\mathbb C)$,
 it implies that  $X(M)$ must be lie in a $2-$plane.
\end{remark}

\section{Immersed harmonic surfaces and induced metric}

\subsection{Immersed harmonic surfaces in ${\mathbb R}^n$}
Consider a  regular  and immersed map $X=(x_1,\cdots,x_n): M\rightarrow {\mathbb R}^n$,
where $M$ is a smooth oriented two-manifold without boundary.
By choosing a  local coordinate
$(u, v)$ and let $z=u+iv$, $M$ can be regarded as a Riemann surface.
The surface $X(M)$ is referred to as {\it an immersed harmonic surface} in  ${\mathbb R}^n$ if the map $X$ is harmonic, meaning that
$$\bigtriangleup X := \left(\frac{\partial^2 x_1}{\partial z\partial \bar{z}}, \cdots,\frac{\partial^2 x_n}{\partial z\partial \bar{z}}\right)\equiv \mathbf{0}$$
where ${\partial \over \partial z} = {1\over 2} (\partial/\partial u - \sqrt{-1} \partial/\partial v)$ and
${\partial \over \partial \bar{z}} = {1\over 2} (\partial/\partial u + \sqrt{-1} \partial/\partial v)$.\par

\medskip
Set $\phi_k={\partial x_k \over \partial z}$ for $k=1,\cdots,n$.
The map $X$ is harmonic if and only if
\begin{equation*}
\phi: ={\partial X\over \partial z}   =(\phi_1, \cdots, \phi_n)
\end{equation*}
is holomorphic.
Although $\phi_k$ are only locally defined, the holomorphic one-forms $\Phi_k: = \phi_k dz$ can be globally defined on $M$.
Consequently,
the holomorphic map  $G: = [\Phi_1:\cdots:\Phi_n]: M\rightarrow {\mathbb P}^{n-1}({\mathbb C})$ is well-defined
and we call it the {\it generalized Gauss map} of the harmonic
surface $X(M)$ (also see \cite{Chen-2021}).

\subsection{Induced metric}
Let $ds^2$ be the metric on $M$  induced by $X$ from the standard  inner product on ${\Bbb R}^n$.
In term of local coordinate $(u,v)$, the first fundamental form of $ds^2$ is given by
\begin{equation}\label{HST1}
I=ds^2=Edu^2 +2Fdudv+Gdv^2
\end{equation}
with \begin{equation*}
E=X_u\cdot X_u, \ F=X_u\cdot X_v,\ G=X_v\cdot X_v.
\end{equation*}
By utilizing the complex local coordinate $(z, \bar z)$, we can rewrite (\ref{HST1}) as the following form
\begin{equation}\label{pro1.1.}
ds^2=hdz^2+2\|\phi\|^2|dz|^2+\overline{hd z^2},
\end{equation}
where
$h=\phi\cdot \phi=\sum_{k=1}^n\phi_k^2.$
The regularity of $X$ ensures $\|\phi\|^2:=\sum_{k=1}^n|\phi_k|^2\neq 0$.
Hence,
the induced metric $ds^2$ can be decomposed into two parts:
the {\it associated conformal metric} $\Gamma:=2 \|\phi\|^2 |dz|^2$ (usually known as the {\it Klotz metric}) and
the {\it Hopf differential} $hdz^2$.
It is evident  that $|h|<\|\phi\|^2$ and
\begin{equation}\label{compar-metric}
ds^2\leq 4\|\phi\|^2|dz|^2.
\end{equation}
We observed that if $ds^2$ is complete, then the associated conformal metric $\Gamma$ is also complete (see also in \cite[ Lemma 1]{KT-1976}).
However, the converse is not necessarily true.
In the following, the immersion $X$ is said to be {\it weakly complete} if the
 associated conformal metric  $\Gamma$ is complete.
In particular, if $h=0$,
then we have
$$ds^2=\Gamma, ~~E=G, ~~F=0.$$
This implies $(u,v)$ is an isothermal system,
and $X(M)$ becomes a minimal surface immersed in $\mathbb{R}^3$.
In other words, $M$ is considered as a Riemann surface with the conformal metric $ds^2$
by associating a holomorphic local coordinate $z=u+iv$ for each  isothermal system $(u,v)$.

\subsection{Estimate of Gauss curvature}
We now derive the  expression for the  curvature of a harmonic immersed surface in $\mathbb{R}^3$.
Let $\mathcal{K}_{ds^2}$ denote the intrinsic curvature of the induced metric $ds^2$ and $\mathcal{K}_\Gamma$ denote the Gauss curvature with respect to the Klotz metric $\Gamma:=2\|\phi\|^2 |dz|^2$.
\begin{definition}\cite{Fujimoto-1993}\label{def-3}
Let $M$ be a Riemann surface with a metric $\Gamma$
which is conformal, namely, represented as
$$\Gamma=\rho_z^2|dz|^2$$
with a positive $C^\infty$ function $\rho_z$
in term of a holomorphic  local coordinate $z$.
For each point $p\in M$, the Gauss curvature of $M$
at $p$ is given by
$$\mathcal{K}_{\Gamma}:=-\frac{\Delta_z\log \rho_z}{\rho_z^2}.$$
\end{definition}
Applying the definition to the Klotz metric $\Gamma:=2\|\phi\|^2 |dz|^2$, we obtain
\begin{equation}\label{equ-11}
\mathcal{K}_\Gamma=-\frac{\Delta\log \sqrt{2}\|\phi\|}{2\|\phi\|^2}
=-\frac{\|{\phi}'\|^2\|{\phi}\|^2-|{\phi}'\cdot\overline{{\phi}}|^2}{(\| {\phi}\|^2)^3}.
\end{equation}
Now, let's consider
the second fundamental form, which can be expressed as
$$II(\mathfrak{n})= Ldu^2+2Mdudv+Ndv^2,$$
with
\begin{equation*}
L=X_{uu}\cdot \mathfrak{n}, ~M=X_{uv}\cdot \mathfrak{n}, ~N=X_{vv}\cdot \mathfrak{n}.
\end{equation*}
Furthermore, we have
\begin{eqnarray*}
X_u&=&{\phi}+\overline{\phi},\ \ X_v=\sqrt{-1}(\phi-\overline{\phi})\\
X_{uu}&=&{\phi}'+\overline{{\phi}'},\ \  X_{uv}=\sqrt{-1}({\phi}'-\overline{{\phi}'}), \ \ X_{vv}=-({\phi}'+\overline{{\phi}'})
\end{eqnarray*}
and
\begin{equation*}
\mathfrak{n}=\frac{X_u\times X_v}{\|X_u\times X_v\|}=\frac{\sqrt{-1}(\overline{{\phi}}\times {\phi})}{\sqrt{\|{\phi}\|^4-|h|^2}}.
\end{equation*}
We have the following expression for $\mathcal{K}_{ds^2}$:
\begin{eqnarray}\label{equ-20}
\mathcal{K}_{ds^2}=\frac{LN-M^2}{EG-F^2}=-4\frac{|X_u\times X_v\cdot \overline{{\phi}'}|^2}{(EG-F^2)^2}=-\frac{|\bar{\phi}\times{{ \phi}}\cdot{\phi}'|^2}{(\|{\phi}\|^4-|h|^2)^2}.
\end{eqnarray}
 This expression relates the intrinsic curvature
$\mathcal{K}_{ds^2}$
to the cross product of the partial derivatives of the immersion $X$, denoted by $X_u$ and $X_v$,
and the conjugate of the derivative of $\phi$.
It also involves the metric coefficients
$E,F,G$ and the norm of the conformal factor $\|\phi\|$.
According to Lemma 1 in \cite{KT-1980},
there exists a positive function $\mu$ such that
\begin{equation*}
\mathcal{K}_\Gamma \leq \mu \mathcal{K}_{ds^2}\leq0.
\end{equation*}

\section{Estimate of Gauss curvature for $K$-QC harmonic surfaces}
In this section, we will discuss K-quasiconformal harmonic surfaces in $\mathbb{R}^3$.
\subsection{$K$-QC harmonic surfaces in $\mathbb{R}^3$}
Let $M$ be an open Riemann surface and  $X=(x_1,x_2,x_3):M \rightarrow \mathbb{R}^3$ be a
  harmonic immersion.
We can express the norm of the gradient of $X$ as
\begin{equation}\label{HST4}
\|\nabla X\|^2=E+G=4\|\phi\|^2,
\end{equation}
where $\|\nabla X\|^2$ is the Hilbert-Schmidt norm defined by
$\|\nabla X\|^2:=\|X_u\|^2+ \|X_v\|^2$.
Also the Jacobian of $X$ is given by
\begin{equation}\label{HST5}
 J_X=\|X_u\times X_v\|=\sqrt{EG-F^2}=2\sqrt{\|\phi\|^4-|h|^2}.
\end{equation}
An immersion $X=(x_1,x_2,x_3): M\rightarrow \mathbb{R}^3$ is
called $K$\textit{-quasiconformal}(for short \textit{$K$-QC}) if it  satisfies the  inequality
\begin{equation}\label{QCD1}
||\nabla X||^2\leq \left(K+\frac{1}{K}\right)J_X,
\end{equation}
which is equivalent to
\begin{equation}\label{QCD2}
\|\phi\|^2\leq \frac{K^2+1}{2K}\sqrt{\|\phi\|^4-|h|^2}.
\end{equation}
A Riemann surface $M$ that admits a $K$-quasiconformal harmonic immersion $X$ into $\mathbb{R}^3$ is referred to as a
$K$-\textit{quasiconformal harmonic surface}.
It is important to note that we adopt the definition of quasiconformality given by D. Kalaj \cite{KD-2013} (see also \cite{AL-2013}).

\subsubsection{Minimal surfaces immersed in $\mathbb{R}^3$}
If we set $K=1$ in (\ref{QCD2}), we obtain $h=0$, which leads to  the following two relations
\begin{equation*}{}
\|X_u\|=\|X_v\| \ \ \text{and}\ \ X_u\cdot X_v=0.
\end{equation*}
In this case,
 we say that $X$ is an isothermal parametrization (isothermal coordinate) of the surface $M$.
 Therefore,
 minimal surfaces can be considered as $1$-quasiconformal harmonic surfaces.

\subsection{The comparison between $ds^2$ and $\Gamma$ in $K$-QC case}
For the $K$-quasiconformal harmonic immersion $X$,
we have the following
the relation (see \cite[Lemma 4.4]{Chen-2021})
\begin{equation*}
\frac{2}{K^2+1}\Gamma\leq ds^2\leq \frac{2K^2}{K^2+1}\Gamma.
\end{equation*}
This inequality provides a comparison between the two metrics, indicating that the induced metric $ds^2$
is bounded by a constant multiple of the conformal metric
$\Gamma$.

\subsection{The relationship between two Gauss maps of $K$-QC harmonic surfaces}
For a harmonic surface in $\mathbb{R}^3$,
both the classical Gauss map $\mathfrak{n}$(i.e. unit normal map) and
the generalized Gauss map $G$ can be defined.
The author and collaborators \cite{Chen-2021} further showed
the  relationship between  $\mathfrak{n}$ and $G$ as following:
\begin{lemma}\cite{Chen-2021}\label{lem1}
Let $M$ be a $K$-quasiconformal harmonic surface in $\mathbb{R}^3$. Then for any unit normal vector $\mathfrak{n}$  and   unit vector $\mathfrak{b}$ at a  point $p$ on $M$, we have the following inequality:
\begin{equation}\label{pro1.8}
\frac{K^2+1}{2K^2}\frac{|\phi\cdot \mathfrak{b}|^2}{\|\phi\|^2}\leq \frac{1-|\mathfrak{n}\cdot \mathfrak{b}|^2}{2}
\leq \frac{K^2+1}{2}\frac{|\phi\cdot \mathfrak{b}|^2}{\|\phi\|^2}.
\end{equation}
In particular, when $K=1$, the inequality simplifies to:
\begin{equation}\label{pro1.10}
\frac{1-|\mathfrak{n}\cdot \mathfrak{b}|^2}{2}=\frac{|\phi\cdot \mathfrak{b}|^2}{\|\phi\|^2}.
\end{equation}
\end{lemma}
The above result implies that
 for $K$-quasiconformal harmonic surface,
a normal vector $\mathfrak{n}$ makes an angle of at least $\theta$ with a given vector $\mathbf{d}$ if and only if its generalized Gauss map $G$ has a positive projective distance to a hyperplane $H$ with the unit normal $\mathbf{d}$.
It should be noted that
the additional condition of  $K$-quasiconformality is necessary.
In fact,
the conclusion does not hold for certain harmonic immersed surfaces,
as demonstrated by the counterexample provided by A. Alarc\'{o}n and F. J. L\'{o}pez in \cite{AL-2013}.\par

\medskip
To prove  Theorem \ref{mainthm-2}, we need the following result.
\begin{lemma}\cite[Theorem 2]{Chen-2021-MZ}\label{lem-conformal}
Let $M$ be an open Riemann surface and $G:M\rightarrow {\mathbb P}^{n}(\mathbb{C})$ be a nonconstant holomorphic map.  Consider the conformal metric  on $M$ given by
$$\Gamma=\|\tilde{G}\|^{2m}|\omega|^2,$$
where $\tilde{G}$ is a reduced representation of $G$,
 $\omega$ is a holomorphic 1-form, and  $m\in {\mathbb N}$.
Assume that $G$ omits more than $\frac{n+1}{2}(mn+2)$ hyperplanes in ${\mathbb P}^{n}(\mathbb{C})$ located in general position.
Then there exists a constant $C$, which depends only on the set of omitted hyperplanes,  such that  the following inequality holds for all $p\in M$:
\begin{eqnarray*}
|\mathcal{K}_\Gamma(p)|^{\frac{1}{2}}d_\Gamma(p)\leq C,
\end{eqnarray*}
where $\mathcal{K}_\Gamma(p)$ is the Gauss curvature of $M$ at $p$ with respect to the metric $ds^2$,
and $d_\Gamma(p)$ is the geodesic distance from $p$ to the boundary of $M$.
\end{lemma}

\begin{proof}[The proof of Theorem \ref{mainthm-2}]
Let $X:M\rightarrow \mathbb{R}^3$ be a $K$-quasiconformal harmonic
surface, where $M$ is an open Riemann surface.
For each direction $\mathbf{d}=(d_1,d_2,d_3)\in S^2$,
it corresponds to a hyperplane $H_\mathbf{d}:=\{[z_1:z_2:z_3]|d_1z_1+d_2z_2+d_3z_3=0\}$.
Given the condition that the normal $\mathfrak{n}$ omits $7$ directions $\mathbf{d}_1,\cdots,\mathbf{d}_7$ in $S^2$
such that no three of them
lie in a plane in ${\mathbb R}^3$,
Lemma \ref{lem1} guarantees the existence of
 $7$ hyperplanes $H_1,\cdots,H_7$  in general position such that the generalized Gauss map $G$ omits these hyperplanes.
In  the following, we consider an open Riemann surface $M$ equipped with a conformal metric $\Gamma=2\|\phi\|^2|dz|^2$.
Let $\mathcal{K}_{ds^2}$ and $\mathcal{K}_\Gamma$ denote the
the  curvature  of the induced metric $ds^2$  and  the Klotz metric $\Gamma$ respectively.
By Lemma \ref{lem-conformal},
there exists a positive constant $C$ such that
\begin{equation}\label{equ-22}
|\mathcal{K}_\Gamma(p)|d_\Gamma(p)^2\leq C.
\end{equation}
On the other hand,
using (\ref{equ-20}) and (\ref{QCD2}),
we can deduce that
$$|\mathcal{K}_{ds^2}|=\frac{|\bar{\phi}\times{{ \phi}}\cdot{\phi}'|^2}{(\|{\phi}\|^4-|h|^2)^2}
\leq\frac{|\bar{\phi}\times{{ \phi}}\cdot{\phi}'|^2}{\|\phi\|^8}\left(\frac{K^2+1}{2K}\right)^4.$$
Together with (\ref{equ-11}),
we can derive the following inequality:
\begin{eqnarray*}
\left|{\mathcal{K}_{\Gamma}\over \mathcal{K}_{ds^2}}\right|&\geq&
\frac{\|{\phi}'\|^2\|{\phi}\|^2-|{\phi}'\cdot\overline{{\phi}}|^2}{|\bar{\phi}\times{{ \phi}}\cdot{\phi}'|^2}\cdot\|{\phi}\|^2\left({2K\over K^2+1}\right)^4\\
&\geq& \frac{\|{\phi}'\|^2\|{\phi}\|^2-|{\phi}'\cdot\overline{{\phi}}|^2}{\|\phi'\times{{ \bar{\phi}}}\|^2}\cdot\left({2K\over K^2+1}\right)^4.
\end{eqnarray*}
Using the Lagrange identities in complex form,
we have
$$\|{\phi}'\|^2\|{\phi}\|^2-|{\phi}'\cdot\overline{{\phi}}|^2=\|\phi'\times{{ \bar{\phi}}}\|^2.$$
Hence, we obtain
$$\left|{\mathcal{K}_{\Gamma}\over \mathcal{K}_{ds^2}}\right|\geq \left({2K\over K^2+1}\right)^4,$$
which implies that
$|\mathcal{K}_{ds^2}(p)|\leq\left({K^2+1\over 2K}\right)^4|\mathcal{K}_\Gamma(p)|$ holds for any $p\in M$.
By (\ref{compar-metric}),
we have $d(p)\leq \sqrt{2}d_\Gamma(p)$.
Therefore,  from (\ref{equ-22}), we can deduce that  for all $p\in M$,
$$|\mathcal{K}_{ds^2}(p)|d(p)^2\leq \left(\frac{K^2+1}{2K}\right)^4|\mathcal{K}_\Gamma(p)|\cdot 2d_\Gamma(p)^2\leq 2C,$$
where $C$ is a positive constant depending only on
$K$, $\mathbf{d}_1,\cdots,\mathbf{d}_7$.

\end{proof}

\section{Modified defect relation for the generalized Gauss map of harmonic surfaces in $\mathbb{R}^n$}
The classical defect relation in value distribution theory of
meromorphic functions is a well-known result.
It is stated as follows:
\begin{theorem}\label{classical defect}
Let $F$ be a nondegenerate holomorphic map of $\mathbb{C}$ into $\mathbb{P}^n(\mathbb{C})$.
Then
$$\sum_{1\leq j\leq q}\delta_F(H_j)\leq n+1$$
for arbitrary hyperplanes $H_1,\cdots,H_q$ in general position.
\end{theorem}
In our context, we aim to establish a modified defect relation for the generalized Gauss map of harmonic surfaces in $\mathbb{R}^n$. To do so, we utilize the results obtained in the previous sections.
\subsection{A new type of modified defect}
In \cite{CO-1967}, S. S. Chern and R. Osserman established an important result regarding the Gauss map of non-flat complete minimal surfaces in $\mathbb{R}^n$.
They showed that this Gauss map intersects a dense set of
hyperplanes. H. Fujimoto \cite{Fujimoto-1983} further improved the understanding of the Gauss map by proving that for non-degenerate minimal surfaces in $\mathbb{R}^n$, the Gauss map fails to intersect at most $n^2$ hyperplanes in general position.
Expanding on Fujimoto's work, in \cite{Fujimoto-1990} He introduced some new types of
modified defects and provided a defect relation for a holomorphic map from
a Riemann surface into $\mathbb{P}^n(\mathbb{C})$.
\begin{definition}\cite{Fujimoto-1990}\label{def-1}
The new modified defect of $H_j$ for $F$ can be defined by
$$\delta^H_F(H_j):=1-\inf\{\eta\geq 0; \eta {\mbox ~satisfies~ condition} (*)\}.$$
Here, condition $(*)$ means that there exists a $[-\infty, \infty)-$valued continuous function
$\mu$ on $M$ which is harmonic on $M\setminus \{z:F(H_j)(z)=0\}$ and satisfies the following conditions
\begin{enumerate}
 \item[$(H1)$] $e^\mu\leq \|F\|^\eta$, {\mbox where} $\|F\|:=\|F_0\|=(|f_0|^2+\cdots+|f_k|^2)^{1/2}$,
 \item[$(H2)$] for each $z_0\in\{z:F(H_j)(z)=0\}$  there exists the limit
 \begin{equation*}
\lim_{z\rightarrow z_0}(\mu(z)-\min\{\nu_{F(H_j)},k\}\log|z-z_0|)\in[-\infty, \infty),
 \end{equation*}
where $z$ is a holomorphic local coordinate around $z_0$.
\end{enumerate}
\end{definition}
It was shown in  \cite{Fujimoto-1989,Fujimoto-1990}
that the modified defects satisfy $0\leq\delta^H_F(H_j)\leq1$.
Notably,
if $F(H_j)$ has no zero, i.e., $F$ omits the hyperplane $H_j$, then $\delta^H_F(H_j)=1$.
Focusing specifically on minimal surfaces, H. Fujimoto derived the following result.
\begin{theorem}\cite{Fujimoto-1990}
Let $M$ be a complete minimal surface in $\mathbb{R}^n$, and $G$ be the Gauss map of $M$.
If $G$ is non-degenerate, then
$$\sum_{1\leq j\leq q}\delta_G^H(H_j)\leq \frac{n(n+1)}{2}$$
for arbitrary hyperplanes $H_1, \cdots, H_q$ in general position.
\end{theorem}
These findings contribute to our understanding of the behavior of the Gauss map for minimal surfaces in $\mathbb{R}^n$, shedding light on the intersection patterns with hyperplanes in general position.

\subsection{Derived curves}
Consider a holomorphic map $F: \Delta_R\rightarrow{\mathbb P}^k(\mathbb{C})$, where $\Delta_R:=\{z||z|<R\}\subset\mathbb{C}, 0<R\leq \infty$.
We ssume that $F$ is  linearly non-degenerate.
Let $f=(f_0,\cdots,f_k)$ be  a reduce representation of $F$.
Define
$f^{(s)}=(f^{(s)}_0,\cdots,f^{(s)}_k)$
and
\begin{eqnarray*}
\tilde{F}_s=f^{(0)}\wedge\cdots\wedge f^{(s)}:\Delta_R\rightarrow\bigwedge^{s+1}\mathbb{C}^{k+1},
\end{eqnarray*}
for each $s=0,\cdots,k$.
Obviously, $F_{k+1}\equiv0$.
Let ${\mathbb P}:\bigwedge^{s+1}\mathbb{C}^{k+1}\setminus\{0\}\rightarrow{\mathbb P}^{C_{k+1}^{s+1}-1}(\mathbb C)$ be the canonical  projection map,
and define $F_s={\mathbb P}(\tilde{F}_{s})$.
This map $F_s$ is referred to as the {\it $s$-th  derived curve of  $F$}.\par

\medskip
Let $\{e_0,\cdots,e_k\}$ be the standard basis of $\mathbb{C}^{k+1}$.
For $0\leq s\leq k$,  we can express
$$\tilde{F}_s = \sum_{0\leq i_0<\cdots<i_s\leq k} W(f_{i_0},\cdots,f_{i_s})  e_{i_0}\wedge \cdots \wedge e_{i_s},$$
where $W(f_{i_0},\cdots,f_{i_s})$ denotes the Wronskian of  $f_{i_0}, \cdots, f_{i_s}$.
Consequently, we have
$$\|\tilde{F}_s\|^2:=\sum_{0\leq i_0<\cdots<i_s\leq k}|W(f_{i_0},\cdots,f_{i_s})|^2.$$
Obviously, $\|\tilde{F}_s\|\not\equiv 0$ for
 $0\leq s\leq k$ under the assumption of linear non-degeneracy.\par

\medskip
For a hyperplane $H_j$ in ${\mathbb P}^k(\mathbb{C})$ with the unit  normal vector $\textbf{a}_j=(a_{j0}, \cdots, a_{jk})$, we define, for $0\leq s\leq k$,
\begin{equation}\label{derived-equ-10}
\|F_s(H_j)\|^2=\|(\tilde{F}_s,\textbf{a}_j)\|^2:=\sum_{0\leq i_1<\cdots<i_s\leq k}\left|\sum_{t\neq i_1, \cdots, i_s}a_{jt}W(f_t,f_{i_1},\cdots,f_{i_s})\right|^2.
\end{equation}
In particular,
$$\|F(H_j)\|=\|F_0(H_j)\|=|a_{j0}f_0+\cdots+a_{jk}f_k|,$$
and for any $H_j$,
$$\|F_k(H_j)\|=\|\tilde{F}_k\|=|W(f_{0},\cdots,f_{k})|.$$
If $F$ is linearly non-degenerate,
then  it can be shown that $\|F_s(H_j)\|\not \equiv 0$ for all $0\leq s\leq k$(see \cite{Chen-2021} for details).
We may assume that $\|\mathbf{a}_j\|=1$, and denote the distance between $F_s$ and $H_j$ by
$$\varphi_s(\mathbf{a}_j)=\frac{\|F_s(H_j)\|^2}{\|F_s\|^2}.$$

\subsection{Auxiliary results}
\begin{lemma}\cite{Chen-1990,Nochka-1983}\label{lem-3}
Let $\{H_j\}_{j=1}^q$ be a set of hyperplanes in ${\mathbb P}^k(\mathbb{C})$ in $n$-subgeneral position,
where $q>2n-k+1$.
Then there exist some constants $\varpi(j)$ and $\theta>0$ such that:
\begin{enumerate}
\item[$\bullet$] $0<\varpi(j)\theta\leq 1$ for all $1\leq j\leq q$,
\item[$\bullet$] $q-2n+k-1=\theta(\sum_{j=1}^q\varpi(j)-k-1)$,
\item[$\bullet$] $1\leq (n+1)/(k+1)\leq \theta\leq (2n-k+1)/(k+1)$,
\item[$\bullet$] if $B\subset\{1,\cdots,q\}$ and $\# B\leq n+1$, then $\displaystyle \sum_{j\in B}\varpi(j)\leq \dim span\{\textbf{a}_j\}_{j\in B}$
\end{enumerate}
\end{lemma}
Here $\varpi(j)$ are called the Nochka weights associated to the hyperplanes $H_j(1\leq j\leq q)$.
It is clear that if  hyperplanes $H_1, \dots, H_q$ in ${\mathbb P}^n(\mathbb{C})$ are in general position, then, for $k\leq n$ and considering  ${\mathbb P}^k(\mathbb{C})\subset  {\mathbb P}^n(\mathbb{C})$,
the restricted hyperplanes $H_1\cap {\mathbb P}^k(\mathbb{C}), \dots, H_q\cap {\mathbb P}^k(\mathbb{C})$ are
in $n$-subgeneral position.
\begin{lemma}\cite{Chen-1990,Nochka-1983}\label{Nochka}
Consider the assumptions stated in  Lemma \ref{lem-3}.
Let $\{E_j\}_{j=1}^q$ be a sequence of real numbers with $E_j>1$ for all $j$.
For any subset $B\subset\{1,\cdots,q\}$ with  $0<\# B\leq n+1$,
there exists a subset $B_1\subset B$ such that $\{\textbf{a}_j\}_{j\in B_1}$ forms a basis for  the linear space spanned by $\{\textbf{a}_j\}_{j\in B}$
and
$$\prod_{j\in B}E_j^{\varpi(j)}\leq \prod_{j\in B_1}E_j.$$
\end{lemma}

\begin{lemma}\cite[Lemma 3.2.13]{Fujimoto-1993}\label{lem-13}
Let $F:\Delta_R\rightarrow \mathbb{P}^k(\mathbb{C})$ be a nondegenerate holomorphic  map in $\mathbb{P}^k(\mathbb{C})$
with a reduce representation $F=(f_0,\cdots,f_k)$.
Consider hyperplanes $H_1, \cdots, H_q$ in  $\mathbb{P}^k(\mathbb{C})$ in $n$-subgeneral position and
let $\varpi(1),\cdots,\varpi(q)$ be the Nochka weights
associated with these hyperplanes, where $q>2n-k+1$.
Define
$$D=\frac{|W(f_0,\cdots,f_k)|}{|F(H_1)|^{\varpi(1)}\cdots|F(H_q)|^{\varpi(q)}}.$$
Then
$$\nu_D+\sum_{j=1}^q\varpi(j)\min(\nu_{F(H_j)},k)\geq 0.$$
\end{lemma}

\begin{lemma}\cite{Ahlfors-1973}\label{Schwarz}
Let $\chi=\frac{i}{2\pi}\Omega(z)dz\wedge d\bar{z}$ be a continuous pseudo-metric on $\Delta_R$
with  curvature  bounded above by a negative constant.
Then there exists a positive constant $C$ such that
$$\Omega(z)\leq C\cdot\left(\frac{2R}{R^2-|z|^2}\right)^2.$$
\end{lemma}

In this section, we  construct a pseudo-metric on $\Delta_R$
which plays a key role in proving Theorem \ref{mainthm-1}.
\begin{lemma}\label{lem-11}
Let $F:\Delta_R\rightarrow \mathbb{P}^k(\mathbb{C})$ be a nondegenerate holomorphic  map with a reduce representation
$(f_0,f_1,\cdots,f_k)$.
Let $H_1,\cdots,H_q$ be hyperplanes in  $\mathbb{P}^k(\mathbb{C})$ in $n$-subgeneral position
and $\varpi(j)$  be  their Nochka weights.
Assume that there exist positive constants $\eta_j(1\leq j\leq q)$
and $[-\infty,\infty)-$valued continuous subharmonic functions $\mu_j$ satisfying conditions
$(H1)$ and $(H2)$.
Let $N$ be a positive constant, and define
{\small{
\begin{equation}\label{xi}
\displaystyle\xi=\frac{\|F\|^{\sum_{j=1}^q\varpi(j)(1-\eta_j)}}{\|F\|^{(k+1)+\frac{2q}{N}(k^2+2k+1)}}\cdot
\prod_{j=1}^q\left(\frac{e^{\mu_j}}{|F(H_j)|}\right)^{\varpi(j)}
\cdot\frac{\|\tilde{F}_k\|^{1+\frac{2q}{N}}\cdot\prod_{s=0}^{k-1}\|\tilde{F}_s\|
^{\frac{4q}{N}}}{\prod_{s=0}^{k-1}\prod_{j=1}^q(N-\log\varphi_s(\mathbf{a}_j))}
\end{equation}
}}
on $\Delta_R^\ast: =\Delta_R\setminus\{\cup_{1\leq j\leq q, 0\leq s\leq k}\varphi_s(\mathbf{a}_j)=0\}$.
If $\displaystyle\sum_{j=1}^q(1-\eta_j)>2n-k+1$,
then there exists a positive constant $d_k$ such that for large $N$,
$$dd^c\log \xi\geq \frac{1}{2\pi}d_k\xi^{2\kappa}dz\wedge d\bar{z},$$
where $\kappa=\frac{1}{\sum_{s=0}^{k-1}[(k-s)+\frac{2q}{N}(k-s)^2]}$.
\end{lemma}

\begin{proof}
By the setting of $\xi$ in (\ref{xi}),
it follows that
{\small{
\begin{align*}
dd^c\log \xi&=\{\sum_{j=1}^q\varpi(j)(1-\eta_j)-(k+1)-\frac{2q}{N}(k^2+2k+1)\}dd^c\log\|F\|+\sum_{j=1}^q\varpi(j)dd^c\mu_j\\
&-\sum_{j=1}^q\varpi(j)dd^c\log|F(H_j)|+(1+\frac{2q}{N})dd^c\log\|\tilde{F}_k\|+\frac{4q}{N}\sum_{s=0}^{k-1}dd^c\log\|\tilde{F}_s\|\\
&+\sum_{s=0}^{k-1}\sum_{j=1}^qdd^c\log\frac{1}{N-\log\varphi_s(\mathbf{a}_j)}.
\end{align*}
}}
Moreover,
under the given conditions,
each $\mu_j$  is harmonic function on $\Delta_R^\ast$ and
$\tilde{F}_k$ and $F(H_j)$ are  holomorphic functions.
For all $j$,
$$dd^c\mu_j=0, ~dd^c\log\|\tilde{F}_k\|=0, ~dd^c\log|F(H_j)|=0.$$
 hold on $\Delta_R^\ast$.
Consequently,
\begin{align}
dd^c\log \xi&=\{\sum_{j=1}^q\varpi(j)(1-\eta_j)-(k+1)-\frac{2q}{N}(k^2+2k+1)\}dd^c\log\|F\|\nonumber\\
&+\frac{4q}{N}\sum_{s=0}^{k-1}dd^c\log\|\tilde{F}_s\|
+\sum_{s=0}^{k-1}\sum_{j=1}^qdd^c\log\frac{1}{N-\log\varphi_s(\mathbf{a}_j)}.\label{equ-12}
\end{align}
Referring to \cite[Lemma 2.1]{RU-1991}(see also \cite{Cowen-1976, Shabat-1985, Wong-1976}),
we have
\begin{equation}\label{equ-13}
dd^c\log\frac{1}{N-\log\varphi_s(\mathbf{a}_j)}\geq \left\{\frac{\varphi_{s+1}(\mathbf{a}_j)}{\varphi_s(\mathbf{a}_j)(N-\log\varphi_s(\mathbf{a}_j))^2}-\frac{1}{N}\right\}dd^c\log\|\tilde{F}_s\|^2.
\end{equation}
Additionally, utilizing  \cite[Lemma 2.3]{RU-1991} or \cite[Theorem 7.3]{Chen-1987},
we can find a positive constant
 $C_s$ such that
\begin{equation}\label{equ-14}
\sum_{j=1}^q\frac{\varphi_{s+1}(\mathbf{a}_j)}{\varphi_s(\mathbf{a}_j)(N-\log\phi_s(\mathbf{a}_j))^2}\geq
C_s\left(\prod_{j=1}^q\left(\frac{\varphi_{s+1}(\mathbf{a}_j)}{\varphi_{s}(\mathbf{a}_j)}\right)^{\varpi(j)}\frac{1}{(N-\log\varphi_{s}(\mathbf{a}_j))^2}\right)
^{\theta_s}
\end{equation}
holds, where $\theta_s=\frac{1}{k-s+\frac{2q}{N}(k-s)^2}$.
Combining (\ref{equ-12}), (\ref{equ-13}) and (\ref{equ-14}),
we obtain
\begin{align}
dd^c\log \xi&\geq\{\sum_{j=1}^q\varpi(j)(1-\eta_j)-(k+1)-\frac{2q}{N}(k^2+2k)\}dd^c\log\|F\|\nonumber\\
&+\frac{2q}{N}\sum_{s=1}^{k-1}dd^c\log\|\tilde{F}_s\|+\sum_{s=0}^{k-1}Y_s.\label{equ-15}
\end{align}
Here, for $0\leq s\leq k-1$,
\begin{align}\label{equ-16}
Y_s=C_s\left(\prod_{j=1}^q\left(\frac{\varphi_{s+1}(\mathbf{a}_j)}{\varphi_{s}(\mathbf{a}_j)}\right)^{\varpi(j)}\frac{1}{(N-\log\varphi_{s}(\mathbf{a}_j))^2}\right)
^{\theta_s}dd^c\log\|\tilde{F}_s\|^2.
\end{align}
We claim that
\begin{align}\label{equ-18}
dd^c\log \xi\geq\sum_{s=0}^{k-1}Y_s=\frac{i}{2\pi}\sum_{s=0}^{k-1}y_sdz\wedge d\bar{z}
\end{align}
holds for some large $N$.
By Lemma \ref{lem-3}, we obtain
\begin{align}
\sum_{j=1}^q\varpi(j)(1-\eta_j)-(k+1)&=\sum_{j=1}^q\varpi(j)-(k+1)-\sum_{j=1}^q\varpi(j)\eta_j\nonumber\\
&\geq (q-2n+k-1-\sum_{j=1}^q\eta_j)\cdot\frac{k+1}{2n-k+1}.\label{equ-xi-1}
\end{align}
The assumption  $\sum_{j=1}^q(1-\eta_j)>2n-k+1$ implies that the right-hand side of the above inequality is positive,
namely, $\sum_{j=1}^q\varpi(j)(1-\eta_j)-(k+1)>0$.
Therefore, we can choose a sufficiently large positive $N$ such that
\begin{equation}\label{equ-19}
\sum_{j=1}^q\varpi(j)(1-\eta_j)-(k+1)-\frac{2q}{N}(k^2+2k)>0.
\end{equation}
On the other hand, it follows from \cite[Lemma 4.16]{Cowen-1976}(see also \cite{Shabat-1985}) that
\begin{equation}\label{equ-17}
dd^c\log\|\tilde{F}_s\|^2=\frac{i}{2\pi}\frac{\|\tilde{F}_{s-1}\|^2\|\tilde{F}_{s+1}\|^2}{\|\tilde{F}_{s}\|^4}dz\wedge d\bar{z},
\end{equation}
where $0\leq s\leq k-1$ and $\tilde{F}_{-1}\equiv1$.
One further knows each $dd^c\log\|\tilde{F}_s\|^2$ is nonnegative.
Hence,  (\ref{equ-18}) holds.\par

\medskip
For some positive numbers $x_0, \cdots, x_{k-1}$ and $a_0, \cdots, a_{k-1}$,
we apply the elementary inequality
$$\sum_{s=0}^{k-1}a_sx_s\geq \left(\sum_{s=0}^{k-1}a_s\right)\cdot\left(\prod_{s=0}^{k-1}x_s^{a_s}\right)^{\frac{1}{\sum_{s=0}^{k-1}a_s}}.$$
By setting $x_s=\theta_sy_s, a_s=\frac{1}{\theta_s}$,
then from (\ref{equ-16}) and (\ref{equ-17}) there exists a positive constant $d_k$(depending on $c_0,\cdots,c_{k-1},\theta_0,\cdots,\theta_{k-1}$) such that
\begin{align*}
\sum_{s=0}^{k-1}y_s&\geq \left(\sum_{s=0}^{k-1}\frac{1}{\theta_s}\right)
\cdot\left(\prod_{s=0}^{k-1}\theta_s^{\frac{1}{\theta_s}}\right)^{\frac{1}{\sum_{s=0}^{k-1}\frac{1}{\theta_s}}}
\cdot\left(\prod_{s=0}^{k-1}y_s^{\frac{1}{\theta_s}}\right)^{\frac{1}{\sum_{s=0}^{k-1}\frac{1}{\theta_s}}}\\
&=d_k\left(\prod_{j=1}^q\frac{1}{\varphi_0(\mathbf{a}_j)^{\varpi(j)}}\cdot\prod_{s=0}^{k-1}\prod_{j=1}^q\frac{1}{(N-\log\varphi_s(\mathbf{a}_j))^2}
\right)^{\frac{1}{\sum_{s=0}^{k-1}\frac{1}{\theta_s}}}\\
&\cdot\left(\prod_{s=0}^{k-1}\left(\frac{\|\tilde{F}_{s-1}\|^2\|\tilde{F}_{s+1}\|^2}{\|\tilde{F}_{s}\|^4}\right)
^{\frac{1}{\theta_s}}\right)^{\frac{1}{\sum_{s=0}^{k-1}\frac{1}{\theta_s}}}.
\end{align*}
Substituting $\theta_s=\frac{1}{k-s+\frac{2q}{N}(k-s)^2}$ into the last term of above inequality,
we obtain
\begin{align*}
&\prod_{s=0}^{k-1}\left(\frac{\|\tilde{F}_{s-1}\|^2\|\tilde{F}_{s+1}\|^2}{\|\tilde{F}_{s}\|^4}\right)
^{\frac{1}{\theta_s}}\\
&=\frac{\prod_{s=0}^{k-1}\|\tilde{F}_{s-1}\|^{2[(k-s)+\frac{2q}{N}(k-s)^2]}\cdot\prod_{s=0}^{k-1}\|\tilde{F}_{s+1}\|^{2[(k-s)+\frac{2q}{N}(k-s)^2]}}
{\prod_{s=0}^{k-1}\|\tilde{F}_s\|^{4[(k-s)+\frac{2q}{N}(k-s)^2]}}\\
&=\frac{\prod_{s=0}^{k-2}\|\tilde{F}_{s}\|^{2[(k-s-1)+\frac{2q}{N}(k-s-1)^2]}\cdot\prod_{s=1}^{k}\|\tilde{F}_{s}\|^{2[(k-s+1)+\frac{2q}{N}(k-s+1)^2]}}
{\prod_{s=0}^{k-1}\|\tilde{F}_s\|^{4[(k-s)+\frac{2q}{N}(k-s)^2]}}\\
&=\|\tilde{F}_0\|^{-2[k+1+\frac{2q}{N}(k^2+2k-1)]}\cdot\prod_{s=1}^{k-1}\|\tilde{F}_s\|^{\frac{8q}{N}}\cdot|\tilde{F}_k|^{2(1+\frac{2q}{N})}.
\end{align*}
Consequently,
\begin{align*}
\sum_{s=0}^{k-1}y_s
&\geq d_k\left(\frac{\|F\|^{\Sigma_{j=1}^q\varpi(j)}}{\prod_{j=1}^q|F(H_j)|^{\varpi(j)}}
\cdot\frac{1}{\prod_{s=0}^{k-1}\prod_{j=1}^q(N-\log\varphi_s(\mathbf{a}_j))}
\right)^{\frac{2}{\sum_{s=0}^{k-1}\frac{1}{\theta_s}}}\\
&\cdot\left(\|\tilde{F}_0\|^{-[k+1+\frac{2q}{N}(k^2+2k-1)]}\cdot\prod_{s=1}^{k-1}\|\tilde{F}_s\|^{\frac{4q}{N}}\cdot|\tilde{F}_k|^{1+\frac{2q}{N}}\right)
^{\frac{2}{\sum_{s=0}^{k-1}\frac{1}{\theta_s}}}\\
&= d_k(\xi\cdot \prod_{j=1}^q\left(\|F\|^{\eta_j}e^{-\mu_j}\right)^{\varpi(j)})^{\frac{2}{\sum_{s=0}^{k-1}\frac{1}{\theta_s}}}\\
&\geq  d_k\xi^{\frac{2}{\sum_{s=0}^{k-1}\frac{1}{\theta_s}}},
\end{align*}
which ensures that
$$dd^c\log \xi\geq \frac{1}{2\pi}d_k\xi^{2\kappa}dz\wedge d\bar{z}.$$
Here, $\kappa=\frac{1}{\sum_{s=0}^{k-1}[(k-s)+\frac{2q}{N}(k-s)^2]}$.

\end{proof}

\begin{lemma}\label{lem-12}
Under the above assumptions in Lemma \ref{lem-11}.
Let us define
\begin{equation}
\Omega(z)=
\left\{
\begin{array}{ccc}
\xi(z)^{2/\sum_{s=0}^ks(1+\frac{2q}{N}s)},& z\in\Delta_R\setminus\{\cup_{1\leq j\leq q, 0\leq s\leq k}\varphi_s(\mathbf{a}_j)=0\},\\
0,& z\in\{\cup_{1\leq j\leq q, 0\leq s\leq k}\varphi_s(\mathbf{a}_j)=0\}.
\end{array}
\right.
\end{equation}
Then  there exists a positive constant $C$ such
that
$$\Omega(z)\leq C\cdot\left(\frac{2R}{R^2-|z|^2}\right)^2.$$
\end{lemma}

\begin{proof}
First, we  show that $\Omega(z)$ is continuous on $\Delta_R$.
Since $\sum_{j=1}^q(1-\eta_j)>2n-k+1$,
it follows from (\ref{equ-19}) that $\Omega(z)$ is continuous on $\Delta_R\setminus\{\prod_{j=1}^qF(H_j)(z)=0\}$.
Let
\begin{equation}\label{A}
A=\frac{e^{\sum_{j=1}^q\mu_j\varpi(j)}\cdot|\tilde{F}_k|}{|F(H_1)|^{\varpi(1)}\cdots|F(H_q)|^{\varpi(q)}}
\end{equation}
By Lemma \ref{lem-13} and (H2),
we have $\nu_A\geq 0$,
which implies that
$$\lim_{z\rightarrow z_0}\Omega(z)=0,$$
for any zero  $z_0$  of $\prod_{j=1}^qF(H_j)(z)$.
Hence, $\Omega(z)$ is continuous on $\Delta_R$.
By Lemma \ref{lem-11},
letting
$\chi=\frac{i}{2\pi}\Omega(z)dz\wedge d\bar{z}$,
 we see that $\chi$ is a continuous pseudo-metric on $\Delta_R$ with  curvature  bounded above by a negative constant.
Applying  Lemma \ref{Schwarz},   we conclude that there exists a positive constant $C$ such
that
$$\Omega(z)\leq C\cdot\left(\frac{2R}{R^2-|z|^2}\right)^2.$$
\end{proof}

\section{The proof of Theorem \ref{mainthm-1}}
\begin{definition}\cite{Fujimoto-1993}
A  continuous curve $r(t)(0\leq t\leq 1)$ in $M$ is said to be {\it divergent} in $M$
if for each compact set $K$, there is some $t_0$ such that $r(t)\not\in K$ for any $t\geq t_0$.
We define the distance $d(p)(\leq +\infty)$  from a point $p\in M$ to the boundary  of $M$ as
the greatest lower bound of the lengths of all continuous curves which is divergent in $M$.
\end{definition}

\begin{definition}\cite{OR-1961}
A surface with a Riemannian metric is {\it complete}
if every continuous curve which is divergent in $M$ has infinite length.
\end{definition}

Now, we are ready to prove Theorem \ref{mainthm-1}.
\begin{proof}[The proof of Theorem \ref{mainthm-1}]
Let $X=(x_1,\cdots,x_n): M\rightarrow \mathbb{R}^n$ be a complete
harmonic surface with the induced metric
and let $G:M\rightarrow{\mathbb P}^{n-1}({\mathbb C})$ be the Gauss map.
Since the surface $X(M)$ is nonflat,
it follows from \cite[Lemma 3.1]{Chen-2021} that $G$ is nonconstant map.
Hence, we may assume that $G$ is $k-$nondegenrate for some $1\leq k\leq n-1$.
By taking the universal covering surface $M$  if necessary, we can assume that $M$ is simply connected.
According to the uniformization theorem,  $M$ is conformally equivalent to either $\mathbb{C}$ or unit disc $\Delta$.
If $M$ is conformally equivalent to $\mathbb{C}$,
then by the classical defect relation in Theorem  \ref{classical defect},
we have
$\sum_{j=1}^q\delta_G(H_j)\leq k+1$.
Since
$0\leq\delta_G^H(H_j)\leq\delta_G(H_j)\leq 1$ that
Theorem \ref{mainthm-1} holds.
Therefore,
it suffices to consider the case of the unit disc $\Delta$.\par

\medskip
Without loss of generality,
we can regard
 $G$  as a linearly non-degenerate map from
$\Delta$ into ${\mathbb P}^{k}({\mathbb C})$
with a reduced representation $\tilde{G}=(\phi_0,\phi_1,\cdots,\phi_k)$.
For given hyperplanes $H_1,\cdots, H_q\in{\mathbb P}^{n-1}({\mathbb C})$ in general position,
 we can set $\tilde  H_j:= H_j\cap {\mathbb P}^{k}({\mathbb C})$, $1\leq j\leq q$,
such that $\tilde  H_j$ are in $(n-1)$-subgeneral position in ${\mathbb P}^{k}({\mathbb C})$.
Let $\varpi(j)$ be the Nochka weights associated with the hyperplanes $\tilde{H_j}$, $1\leq j\leq q$.
Furthermore,
we may assume that $\tilde{H}_j$ are given by
$$\tilde{H}_j: a_{j0}z_0+a_{j1}z_1+\cdots+a_{jk}z_k=0\quad(1\leq j\leq q).$$
From (\ref{derived-equ-10}),
$G_s(\tilde{H}_j)\not\equiv0$.
Hence, for each $s,j$ there exists $i_1,\cdots,i_s$
such that
$$\xi_{js}=\sum_{t\neq i_1, \cdots, i_s}a_{jt}W(f_t,f_{i_1},\cdots,f_{i_s})$$
does not vanish identically.\par

\medskip
We get a contradiction by  assuming that  $\sum_{j=1}^q\delta_G^H(H_j)>(2n-k-1)(k/2+1)$.
By Definition \ref{def-1},
there exist nonnegative constants $\eta_j(1\leq j\leq q)$ such that
$\sum_{j=1}^q(1-\eta_j)>(2n-k-1)(k/2+1)$ and continuous functions $\mu_j$ on $M$.
Each $\mu_j$ is harmonic on $M\setminus \{z:G(\tilde{H}_j)(z)=0\}$
and satisfies conditions $(H1)$ and $(H2)$.
{\small{
\begin{align*}
\sum_{j=1}^q\varpi(j)(1&-\eta_j)-(k+1)(k/2+1)\\
&=\sum_{j=1}^q\varpi(j)-(k+1)-\sum_{j=1}^q\varpi(j)\eta_j-\frac{k(k+1)}{2}\\
&\geq \left(\sum_{j=1}^q(1-\eta_j)-2(n-1)+k-1\right)\frac{1}{\theta}-\frac{k(k+1)}{2}\\
&\geq \left(\sum_{j=1}^q(1-\eta_j)-2(n-1)+k-1\right)\frac{k+1}{2(n-1)-k+1}-\frac{k(k+1)}{2}\\
&=\frac{k+1}{2n-k-1}(\sum_{j=1}^q(1-\eta_j)-(2n-k-1)(k/2+1))>0.
\end{align*}
}}
Choose some $N$ such that
{\small{
$$\frac{\sum_{j=1}^q\varpi(j)(1-\eta_j)-(k/2+1)(k+1)}{\frac{2}{q}+k^2+2k+1+\sum_{s=0}^ks^2}
<\frac{2q}{N}<
\frac{\sum_{j=1}^q\varpi(j)(1-\eta_j)-(k/2+1)(k+1)}{k^2+2k+1+\sum_{s=0}^ks^2}.$$
}}
Let
\begin{eqnarray*}
\Lambda:=\sum_{j=1}^q\varpi(j)(1-\eta_j)-(k+1)-\frac{2q}{N}(k^2+2k+1)
\end{eqnarray*}
and
\begin{eqnarray*}
\tau: &=&\frac{1}{\Lambda}\left(\frac{1}{2}k(k+1)+
\frac{2q}{N}\sum_{s=0}^{k} s^2\right),
\end{eqnarray*}
which imply that
\begin{eqnarray}\label{equ-4}
0<\tau<1,~ 0<N\Lambda(1-\tau)<4.
\end{eqnarray}
We  define a metric on the set $M^\ast:=\Delta\setminus\{p\in\Delta|\tilde{G}_k\cdot\prod_{j=1}^q\prod_{s=0}^{k-1}|\xi_{js}|=0\}$
as follows:
{\small{
\begin{equation}\label{equ-5}
d\sigma^2=\Omega^2|dz|^2=\left(\frac{\prod_{j=1}^{q}\|G(\tilde{H_j})\|
^{\varpi(j)}}{e^{\sum_{j=1}^q\mu_j\varpi(j)}|\tilde{G}_k|^{1+\frac{2q}{N}}
\prod_{j=1}^{q}(\prod_{s=0}^{k-1}|\xi_{js}|)^{\frac{4}{N}}}\right)^{\frac{2}{(1-\tau)\Lambda}}|dz|^2.
\end{equation}}}
Let $\pi: \hat{M}^\ast\rightarrow M^\ast$ be the universal covering surface,
as showed in \cite{Fujimoto-1989},
by using $\Omega$ and $\pi$
one can construct a single-valued  holomorphic function $w=T(\hat{ z})$ on $\hat{M}^\ast$
such that for $p\in M^\ast$, $T$ maps an open neighborhood $U_{\hat p}$ of $\hat p$
biholomorphically onto an open disc $\Delta_R=\{w\in \mathbb{C}: |w|<R\}$ and $T(\hat p)=0$,
 where $\hat p$ satisfies $\pi(\hat p)=p$.
The least upper bound of  $R>0$
for which $T$ biholomorphically maps some open neighborhood of $\hat p$ onto $\Delta_R$ is denoted as $R_0$.
Thus there exist sequence $\{R_n\}$ converging to $R_0$ and
open neighborhood $\{\hat{U}_n\}$ of $\hat p$ such that $T|_{\hat{U}_n}: \hat{U}_n\rightarrow\Delta_{R_n}$ is biholomorphic.
We define $\hat{U}_0:=\cup_n\hat{U}_n$, and $T$ maps $\hat{U}_0$ onto $\Delta_{R_0}$.
And we also define $\Psi:=\pi\cdot(T|\hat{U}_0)^{-1}$,
which is a local diffeomorphism
of a disk $\Delta_{R_0}=\{w\in \mathbb{C}: |w|<R_0\}$
onto an open neighborhood of $p$ with $\Psi(0)=p$.
 Moreover, $\Psi$ is local isometry, i.e. $\Psi^\ast d\sigma=|dw|$.
By applying  Liouville's theorem, we have $R_0<+\infty$.\par

\medskip
We establish the existence of a point $a_0$ on the boundary of $\Delta_{R_0}$ such that the $\Psi-$image $\Upsilon_{a_0}$
of the line $L_{a_0}=\{w=a_0t: 0<t<1\}$ diverges in $M$.
To prove this, we assume the contrary.
For any $a_0\in\partial\Delta_{R_0}$, there exists a sequence $\{t_n\}$
such that $\lim_{n\rightarrow\infty}t_n=1$
and $\Psi(a_0t_n)\rightarrow z_0\in M$.
We consider two cases:
\begin{enumerate}
\item[$\bullet$] Either the point $z_0$ satisfies $|\tilde{G}_k|\cdot\prod_{j=1}^q\prod_{s=0}^{k-1}|\xi_{js}|=0$.
\item[$\bullet$] Or $z_0$ is an interior of $M^\ast$.
\end{enumerate}
In the first case,
if $\tilde{G}_k(z_0)=0$ or $\xi_{js}(z_0)=0$,
we can use
(\ref{A}) and (\ref{equ-5})
to deduce the existence of  a positive constant $C$ such that
$$\Omega\geq C\cdot\left(\frac{1}{|z-z_0|}\right)^{\frac{4}{N\Lambda(1-\tau)}}$$
in a neighborhood $U_{z_0}$ of $z_0$.
Then, utilizing the setting of $N\Lambda(1-\tau)$ in $(\ref{equ-4})$, we obtain
\begin{eqnarray*}
R_0=\int_{L_{a_0}}|dw|&=&\int_{L_{a_0}}\Psi^\ast d\sigma=\int_{\Upsilon_{a_0}}d\sigma\\
&\geq& \int_{\Upsilon_{a_0}\bigcap U_{z_0}}\frac{C}{|z-z_0|^{\frac{4}{N(1-\tau)\Lambda}}}|dz|=\infty,
\end{eqnarray*}
which contradicts the assumption that $R_0<+\infty$.\par

\medskip
For the second case, where $z_0$ is a interior point  of $M^\ast$,
we further claim that
there exist some $t_0(<1)$ such that $\Psi(a_0t)\in U_{z_0}$ for $t_0<t<1$
if we take a simply connected neighborhood $U_{z_0}$ of $z_0$.
In fact, if there exists $\{t_n^\ast\}_{n=1}^\infty$ such that
$\lim_{n\rightarrow\infty}t_n^\ast=1$ and $\Psi(a_0t_n^\ast)\not\in U_{z_0}$,
then the curve $\Psi(a_0t)(0<t<1)$ goes and returns infinitely often from the boundary of $U_{z_0}$
to a sufficiently small neighborhood of $z_0$.
Since $U_{z_0}$  is relatively compact in $M^\ast$ and
$\Omega$ (the coefficient function of $d\sigma$) is positive continuous,
there exists a positive constant $\epsilon$ such that $\epsilon=\min_{z\in \bar{U}_{z_0}}\Omega(z)$.
By the  local isometry of $\Psi$,
we have
$$R_0=\int_{L_{a_0}}|dw|=\int_{\Upsilon_{a_0}}d\sigma\geq\int_{\Upsilon_{a_0}\cap U_{z_0}}\Omega(z)|dz|=\infty.$$
This  contradicts the assumption that $R_0<\infty$.
Therefore, there must exist some $t_0(<1)$ such that $\Psi(a_0t)\in U_{z_0}$ for $t_0<t<1$.
Note that $U_{z_0}$ can be selected by any small neighborhood  of $z_0$.
Thus, we can conclude that
 $\lim_{t\rightarrow1}\Psi (a_0t)=z_0$.
Let $\hat U_{z_0}$  be a connect component of $\pi^{-1}(U_{z_0})$.
Since $\pi$ is homeomorphism when $\pi$ is restricted to $\hat U_{z_0}$,
there exists a limit $\lim_{t\rightarrow1}(T|\hat U_{z_0})^{-1}(a_0t)=\hat{z}_0$.
Then, $T$ maps an open neighborhood of $\hat{z}_0$ biholomorphically onto a neighborhood of $a_0$,
 ensuring that  $(T|\hat{U}_{z_0})^{-1}$ has a holomorphic extension to a neighborhood of each $a_0$ on the boundary of $\Delta_{R_0}$.
Due to the compactness of the boundary of $\Delta_{R_0}$,
there exists a large $R_\ast(>R_0)$  such that $T$ maps
 an open neighborhood of $\hat{U}_{z_0}$ biholomorphically onto $\Delta_{R_\ast}$.
This contradicts the maximality of $R_0$.
Therefore,  there must exist  a point $a_0$ on the boundary of $\Delta_{R_0}$
such that
the $\Psi-$image $\Upsilon_{a_0}=\Psi(L_{a_0})$ is divergent in $M$.\par

\medskip
To obtain a contradiction for the weak completeness of the surface, we aim to prove the finiteness of the length of
$\Upsilon_{a_0}$
with respect to the associated Klotz metric
$\Gamma=2\|G(z)\|^2|dz|^2$.
Let $l(\Upsilon_{a_0})$ denote the length of  the curve $\Upsilon_{a_0}$ with respect to $\Gamma$.
Then,  we have
\begin{eqnarray}
l(\Upsilon_{a_0})&=&\int_{\Upsilon_{a_0}}\sqrt{2}\|G(z)\||dz|\nonumber\\
&=&\int_{L_{a_0}}\sqrt{2}\|\Psi^\ast G(z)\|\cdot\left|\frac{dz}{dw}\right||dw|.\label{equ-6}
\end{eqnarray}
On the other hand, from (\ref{equ-5}), we obtain:
{\small{
\begin{eqnarray*}
\Psi^\ast d\sigma&=&\Psi^\ast\Omega(z)\cdot\left|\frac{dz}{dw}\right|\cdot|dw|\\
&=&\Psi^\ast\left(\frac{\prod_{j=1}^{q}\|G(\tilde{H_j})\|^{\varpi(j)}}{e^{\sum_{j=1}^q\mu_j\varpi(j)}|\tilde{G}_k|^{1+\frac{2q}{N}}
\prod_{j=1}^{q}(\prod_{s=0}^{k-1}|\xi_{js}|)^{\frac{4}{N}}}\right)^{\frac{1}{(1-\tau)\Lambda}}\cdot\left|\frac{dz}{dw}\right|\cdot|dw|
\end{eqnarray*}
}}
By referring to \cite[Proposition 2.1.6]{Fujimoto-1993},
we can conclude that for any $0\leq s\leq k$,
the following holds:
\begin{equation*}
W_z(\phi_{i_0},\cdots,\phi_{i_s})=W_w(\phi_{i_0},\cdots,\phi_{i_s})\cdot\left(\frac{dw}{dz}\right)^{s(s+1)/2}.
\end{equation*}
Set
\begin{align*}
f_s(w)&=\phi_s(\Psi(w)), F(w)=(f_0:\cdots:f_k)\\
F(\tilde{H}_j)(w)&=G(\tilde{H}_j)(z(w)), F_k=W_w(f_0,\cdots,f_k)\\
\psi_{js}&=\sum_{t\neq i_1, \cdots, i_s}a_{jt}W(f_t,f_{i_1},\cdots,f_{i_s}).
\end{align*}
Hence,
\begin{eqnarray*}
\Psi^\ast d\sigma=\left(\frac{\prod_{j=1}^{q}\|F(\tilde{H_j})\|^{\varpi(j)}}{e^{\sum_{j=1}^q\mu_j\varpi(j)}|F_k|^{1+\frac{2q}{N}}
\prod_{j=1}^{q}(\prod_{s=0}^{k-1}|\psi_{js}|)^{\frac{4}{N}}}\right)
^{\frac{1}{(1-\tau)\Lambda}}\cdot\left|\frac{dz}{dw}\right|^{\frac{1}{1-\tau}}\cdot|dw|.
\end{eqnarray*}
Since $\Psi$ is local isometry, i.e. $\Psi^\ast d\sigma=|dw|$,
we obtain
\begin{equation}\label{equ-8}
\left|\frac{dw}{dz}\right|=\left(\frac{\prod_{j=1}^{q}\|F(\tilde{H_j})\|^{\varpi(j)}}{e^{\sum_{j=1}^q\mu_j\varpi(j)}|F_k|^{1+\frac{2q}{N}}
\prod_{j=1}^{q}(\prod_{s=0}^{k-1}|\psi_{js}|)^{\frac{4}{N}}}\right)
^{\frac{1}{\Lambda}}
\end{equation}
From (\ref{equ-6}), we can derive the following inequality:
{\small{
\begin{align*}
l(\Upsilon_{a_0})
&=\int_{L_{a_0}}\sqrt{2}\|F(w)\|\cdot\left(\frac{e^{\sum_{j=1}^q\mu_j\varpi(j)}|F_k|^{1+\frac{2q}{N}}
\prod_{j=1}^{q}(\prod_{s=0}^{k-1}|\psi_{js}|)^{\frac{4}{N}}}{\prod_{j=1}^{q}\|F(\tilde{H_j})\|^{\varpi(j)}}\right)
^{\frac{1}{\Lambda}}|dw|\\
&\leq\int_{L_{a_0}}\sqrt{2}\left(\frac{\|F(w)\|^\Lambda\cdot e^{\sum_{j=1}^q\mu_j\varpi(j)}|F_k|^{1+\frac{2q}{N}}
\prod_{j=1}^{q}(\prod_{s=0}^{k-1}\|F_{s}(\tilde{H}_j)\|)^{\frac{4}{N}}}{\prod_{j=1}^{q}\|F(\tilde{H_j})\|^{\varpi(j)}}\right)
^{\frac{1}{\Lambda}}|dw|.
\end{align*}}}
Since $\displaystyle \lim_{y\rightarrow0^+}y^{2/N}\log(N-\log y)=0$,
then there exists a positive $C$ such that
$$\frac{\|F_s(\tilde{H}_j)\|^{4/N}}{\|{F}_s\|^{4/N}}=\varphi_s(\mathbf{a}_j)^{2/N}\leq \frac{C}{\log(N-\log \varphi_s(\mathbf{a}_j))}.$$
By applying Lemma \ref{lem-11} and Lemma \ref{lem-12},
we conclude that
{\small{
\begin{align*}
l(\Upsilon_{a_0})
&\leq\int_{L_{a_0}}C_0\left(\frac{\|F(w)\|^\Lambda\cdot e^{\sum_{j=1}^q\mu_j\varpi(j)}|F_k|^{1+\frac{2q}{N}}
\prod_{s=0}^{k-1}\|F_{s}\|^{\frac{4q}{N}}}{\prod_{j=1}^{q}\|F(\tilde{H_j})\|^{\varpi(j)}
\cdot\prod_{j=1}^{q}\prod_{s=0}^{k-1}\log(N-\log \varphi_s(\mathbf{a}_j))}\right)
^{\frac{1}{\Lambda}}|dw|\\
&\leq C_0\int_{L_{a_0}}(\Omega^{1/2})^{\frac{\sum_{s=0}^ks(1+\frac{2q}{N}s)}{\Lambda}}|dw|\\
&\leq C_1\int_{L_{a_0}}\left(\frac{2R_0}{R_0^2-|w|^2}\right)^{\tau}|dw|\\
&\leq C_2\int_{0}^{R_0}\left(\frac{1}{R_0-t}\right)^{\tau}dt,
\end{align*}}}
where $C_i$ are some positive constants.
Note that $0<\tau<1$(as seen in (\ref{equ-4})),
we can further deduce that $l(\Upsilon_{a_0})<+\infty$.
Therefore, we obtain a contradiction that contradicts the weak completeness of the surface $X:M\rightarrow\mathbb{R}^3$.
This completes the  proof of Theorem \ref{mainthm-1}.
\end{proof}



\end{document}